\documentclass[12pt]{amsart}
\usepackage{amsmath,amscd,amssymb,euscript}
\usepackage{latexsym}
\usepackage{graphicx}
\usepackage{tikz}
\usetikzlibrary{backgrounds}
\usetikzlibrary{decorations.fractals}
\usetikzlibrary{calc,intersections,through,backgrounds}
\usepackage{mathrsfs}
\usepackage{epstopdf}
\usepackage{rotating}
\usepackage{lscape}
\input xy
\xyoption{all}
\newcommand{\bC}{{\mathbb C}}

\newcommand{\bP}{{\mathbb P}}

\newcommand{\cJ}{{\mathcal J}}

\newcommand{\cO}{{\mathcal O}}

\newcommand{\cW}{{\mathcal W}}

\newcommand{\ra}{\rightarrow}
\newcommand{\lra}{\longrightarrow}

\theoremstyle{definition}

\newtheorem{proposition}{Proposition}[section]
\newtheorem{lemma}[proposition]{Lemma}

\newtheorem{theorem}[proposition]{Theorem}
\newtheorem{problem}[proposition]{Question}

\newtheorem{definition}[proposition]{Definition}
\newtheorem{corollary}[proposition]{Corollary}

\newtheorem{remark}[proposition]{Remark}

\numberwithin{equation}{section}

\begin{document}

\title{A remark on the generic vanishing of Koszul cohomology }

\author{Jie Wang}

%\address{Department of Mathematics, University of California San Diego, 9500 Gilman Drive \# 0112, La Jolla, CA 92093-0112, USA}

\email{jiewang884@gmail.com}

\subjclass[2010]{}

\begin{abstract}We give a sufficient condition to study the vanishing of certain Koszul cohomology groups for general pairs $(X,L)\in W^r_{g,d}$ by induction. As an application, we show that to prove the Maximal Rank Conjecture (for quadrics), it suffices to check all cases with the Brill-Noether number $\rho=0$.

\end{abstract}

\maketitle

\tableofcontents

\section*{Introduction}
Let $L$ be a base point free $g^r_d$ on a smooth curve $X$, the Koszul cohomology group $K_{p,q}(X,L)$ is the cohomology of the Koszul complex at $(p,q)$-spot
$$\xymatrix{\ar[r]&\wedge^{p+1}H^0(L)\otimes H^0(L^{q-1})\ar[r]^-{d_{p+1,q-1}}&\wedge^pH^0(L)\otimes H^0(L^q)\ar[r]^-{d_{p,q}}&\wedge^{p-1}H^0(L)\otimes H^0(L^{q+1})\ar[r]&}$$
where
$$d_{p,q}(v_1\wedge...\wedge v_p\otimes \sigma)=\sum_{i}(-1)^iv_1\wedge...\wedge\widehat v_i\wedge..\wedge v_p\otimes v_i\sigma.$$ 

Koszul cohomology groups $K_{p,q}(X,L)$ completely determine the shape of a minimal free resolution of the section ring
$$R=R(X,L)=\bigoplus_{k\ge0}H^0(X,L^k).$$
and therefore carry enormous amount of information of the extrinsic geometry of $X$.

In this paper, we are interested in Green's question \cite{Gr}.

\begin{problem}\label{green}
 What do the $K_{p,q}(X,L)$ look like for $(X,L)$ general in $\cW^r_{g,d}$ (i.e. $X$ is a general curve of genus $g$ and $L$ is a general $g^r_d$ on $X$)? 
\end{problem}
The following facts are well known (c.f. \cite{Gr}, \cite{wang14}) for general $(X,L)\in\cW^r_{g,d}$.
\begin{enumerate}
\item We have the following picture of $k_{p,q}=\dim_{\bC}K_{p,q}(X,L)$ (The numbers $k_{p,q}$ in the table are undetermined.):

\begin{table}[h]
    \caption{}
    \begin{tabular}{ccccccccc}
    $0$&$h^1(L)$&$0$&...&$...$&$...$&$0$&$0$&$0$\\
    $0$&$\rho$&$k_{r-2,2}$&...&...&...&$k_{2,2}$&$k_{1,2}$&$k_{0,2}$\\
    $0$&$k_{r-1,1}$&$k_{r-2,1}$&...&...&...&$k_{2,1}$&$k_{1,1}$&$0$\\
    $0$&$0$&$0$&...&...&...&$0$&$0$&$1$\\
        \end{tabular}
    \end{table}
    
    \item \begin{eqnarray}&&k_{p,1}-k_{p-1,2}\nonumber=\chi(\text{Koszul complex})\nonumber\\ \nonumber\\&=&{r+1\choose p}(g-d+r)-{r+1\choose p+1}g+{r-1\choose p}d+{r\choose p+1}(g-1).\nonumber
\end{eqnarray}
    \end{enumerate}

\vspace{.5cm}

Question \ref{green} seems to be too difficult to answer in its full generality.
For the case $p=1$, the Maximal Rank Conjecture (MRC)\footnote{In this paper, we will restrict ourselves to only consider quadrics containing $X$.} \cite{EH2} predicts that the multiplication map 
\begin{eqnarray}\label{multimap}Sym^2H^0(X,L)\stackrel{\mu}\longrightarrow H^0(X,L^2)\nonumber
\end{eqnarray}
 is either injective or surjective, or equivalently
\begin{eqnarray}\label{MRC}\min\{k_{1,1},k_{0,2}\}=0.\nonumber
\end{eqnarray}

%\begin{conjecture} \label{MRC}(Maximal rank conjecture for quadrics) For fixed $g$, $r$, $d$, let $X$ be a general curve of genus $g$ and $|L|$ be a general $g^r_d$ on $X$, then the multiplication map
%\begin{eqnarray}\label{mmap}
%\xymatrix{S^2H^0(X,L)\ar[r]^-{\mu}&H^0(X,L^{2})}
%\end{eqnarray}
%is of maximal rank i.e. either injective or surjective.
% \end{conjecture}
 
% From the definition of Koszul cohomology,
 %$$K_{1,1}(X,L)=\Ker(\mu)=\{\ \text{Quadrics in}\  \mathbb{P}^r\  \text{containing}\ X\ \},$$
%and
%$$K_{0,2}(X,L)=\Coker(\mu),$$
%thus (\ref{MRC}) is just the

Geometrically, this means that the number of quadrics in $\mathbb{P}^r:=\bP(H^0(L))$ containing $X$ is as simple as the Hilbert function of $X\subset\bP^r$ allows.

 There are many partial results about the MRC using the so-called \textquotedblleft m$\acute{e}$thode d'Horace"
originally proposed by Hirschowitz. We refer to, for instance, \cite{BF1}, \cite{BF2} for some recent results in this direction.

For higher syzygies, again there are many results (c.f. \cite{A1}, \cite{A2}, \cite{B}, \cite{Ein}, and \cite{F}). One breakthrough result is Voisin's solution to the generic Green's conjecture \cite{V1} \cite{V2}, which answers Question \ref{green} for the case $L=K_X$.

\begin{definition} For $1\le p\le r-1$, we say property ${\bf GV}(p)^r_{g,d}$ holds if for general $(X,L)\in\cW^r_{g,d}$,
\[
\min\{k_{p,1}(X,L), k_{p-1,2}(X,L)\}=0.
\]
\end{definition}
%For the vanishing of $K_{p,1}$, there is the work of Aprodu \cite{A1} \cite{A2}, which proved the generic version of the Green-Lazarsfeld Gonanity Conjecture. This conjecture predicts that for smooth curve $X$ of gonanity $d$, and $L$ a sufficiently positive line bundle on $X$,
%$$K_{h^0(L)-d,1}(X,L)=0.$$
%Note that Problem \ref{green} does not have any assumption on the positivity of $L$.
\begin{remark}
 The MRC implies that property ${\bf GV}(1)^r_{g,d}$ always holds provided the Brill-Noether number $\rho:=g-(r+1)(g-d+r)\ge0$. However, property  ${\bf GV}(p)^r_{g,d}$ does {\bf not} always hold for $p\ge2$ (c.f Green \cite{Gr} (4.a.2) for more details). 
\end{remark}

In this note, we give a sufficient condition (Theorem \ref{thminduction}) for ${\bf GV}(p)^r_{g,d}$ to imply ${\bf{GV}}(p)^r_{g+1,d+1}$. One could use this to set up an inductive argument for the generic vanishing of Koszul cohomology groups. In each step of the induction, $r$ is fixed and $g$, $d$ go up by $1$.

 In the case $p=1$, this sufficient condition turns out to be an surprisingly simple geometric condition on the quadrics containing the first secant variety $\Sigma_1(X)$ of $X$ (Lemma \ref{lemquadric}). We manage to verify this geometric condition and prove

\begin{theorem}\label{thmmain}The property ${\bf GV}(1)^r_{g,d}$ implies ${\bf GV}(1)^r_{g+1,d+1}$.
\end{theorem}
Based on our knowledge about the base cases of the induction, we have
%\begin{corollary}\label{thmbase}If the MRC holds for all $\rho=0$ cases, then it holds for arbitrary $\rho\ge0$ case.
%\end{corollary} 
\begin{theorem}\label{thmapp}The Maximal Rank Conjecture holds for a general pair $(X,L)\in \cW^r_{g,d}$, if $h^1(L)\le2$.
\end{theorem}

An interesting question remaining is that for $p\ge2$, whether the sufficient condition in Theorem \ref{thminduction} has anything to do with higher syzygies of $\Sigma_1(X)$.

\section{Koszul cohomology on a singular curve}
Throughout this section, let $X=Y\cup Z$ be the union of a smooth curve $Y$ of genus $g$ and $Z=\bP^1$ meeting at two general points $u$ and $v$. Consider a line bundle $L$ (up to $\bC^*$-action) on $X$ such that $A:=L|_Y$ is a $g^r_d$ and $L|_Z=\cO_{\bP^1}(1)$. Note that by construction, every section in $H^0(Y,A)$ 
extends uniquely to a section
in $H^0(X,L)$. Thus we have an isomorphism induced by restriction to $Y$:
\begin{eqnarray}\label{eqnsectionL}
H^0(X,L)\cong H^0(Y,A).
\end{eqnarray}

\begin{proposition}\label{propKp1}Notation as above, if $K_{p,1}(Y,A)=0$, then $K_{p,1}(X,L)=0$.
\end{proposition}
\begin{proof}
Consider the following commutative diagram
$$\xymatrix{\bigwedge^{p+1}H^0(L)\ar[r]\ar[d]^-{\cong}&\bigwedge^{p}H^0(L)\otimes H^0(L)\ar[r]\ar[d]^-{\cong}&\bigwedge^{p-1}H^0(L)\otimes H^0(L^2)\ar[d]\\
\bigwedge^{p+1}H^0(A)\ar[r]&\bigwedge^{p}H^0(A)\otimes H^0(A)\ar[r]&\bigwedge^{p-1}H^0(A)\otimes H^0(A^2)}
$$
where the vertical arrows are restriction maps to $Y$. The hypothesis says that the second row is exact in the middle, a simple diagram chasing gives the conclusion.
\end{proof}

\begin{remark}The argument in Proposition \ref{propKp1} does not generalize to the case $q=2$ because $H^0(Y,A^2)$ is not isomorphic to $H^0(X,L^2)$. 
\end{remark}

To study the relation between $K_{p-1,2}(X,L)$ and $K_{p-1,2}(Y,A)$, we use the duality relation \cite[p. 21]{AN}
$$K_{p-1,2}(Y,A)^{\lor}\cong K_{r-p,0}(Y,A;K_{Y})$$
and compare $K_{r-p,0}(Y,A;K_{Y})$ with $K_{r-p,0}(X,L;\omega_{X})$.
Here $\omega_{X}$ is the dualizing sheaf of $X$. Its restriction $\omega_{X}|_Y\cong K_Y(p+q)$ and $\omega_{X}|_Z\cong \cO_{\bP^1}$. One checks easily that restriction to $Y$ induces the following  isomorphisms:
\begin{eqnarray}&&H^0(X,\omega_{X})\cong H^0(Y,K_Y(u+v)),\\
&&H^0(X,\omega_{X}\otimes L^{-1})\cong H^0(Y, K_Y\otimes A^{-1})%\footnote{We have $\omega_{X}\otimes L^{-1}|_A\cong K_Y\otimes A^{-1}$}
,\\
&&\label{eqnsectionO}
H^0(X,\omega_{X}\otimes L)\cong H^0(Y, K_Y\otimes A(u+v)).
\end{eqnarray}

Denote $M_A$ the kernel bundle associated to a globally generated line bundle $A$, defined by the exact sequence
\[
0\ra M_A\lra H^0(Y,A)\otimes\cO_Y\stackrel{ev}\lra A\lra 0.
\]
Taking $(r-p)$-th wedge product, we obtain
\[
0\lra\wedge^{r-p}M_A\lra\wedge^{r-p}H^0(M)\otimes\cO_Y\lra\wedge^{r-p-1}M_A\otimes A\lra0.
\]

Tensoring the above sequence with $K_Y$, we obtain an isomorphism \cite[Section 2.1]{AN}
\[
H^0(\wedge^{r-p}M_A\otimes K_Y)\cong Ker(\delta_0:\wedge^{r-p}H^0(A)\otimes H^0( K_Y)\lra\wedge^{r-p-1}H^0(A)\otimes H^0(K_Y\otimes A)),
\]
and therefore,
\begin{eqnarray}
K_{r-p,0}(Y,A;K_{Y})\cong \frac{H^0(\wedge^{r-p}M_A\otimes K_Y)}{\wedge^{r-p+1}H^0(A)\otimes H^0(K_Y\otimes A^{-1})}.
\end{eqnarray}

\begin{proposition}\label{propsurff} We have an isomorphism

\[K_{r-p,0}(X,L;\omega_{X})\cong \frac{H^0(\wedge^{r-p}M_A\otimes K_Y(u+v))}{\wedge^{r-p+1}H^0(A)\otimes H^0(K_Y\otimes A^{-1})}.
\] 
\end{proposition}
\begin{proof}Consider the following diagram
\[\small
\xymatrix{\wedge^{\small r-p+1}H^0(L)\otimes H^0(\omega_{X}\otimes L^{-1})\ar[r]^-{d_{-1}}\ar[d]^-{\cong}&\wedge^{r-p}H^0(L)\otimes H^0(\omega_{X})\ar[r]^-{d_0}\ar[d]^-{\cong}&\wedge^{\small r-p-1}H^0(L)\otimes H^0(\omega_{X}\otimes L)\ar[d]^-{\cong}\\
\wedge^{r-p+1}H^0(A)\otimes H^0(K_Y\otimes A^{-1})\ar[r]&\wedge^{r-p}H^0(A)\otimes H^0(K_Y(u+v))\ar[r]^-{\delta_0}&\wedge^{r-p-1}H^0(A)\otimes H^0(K_Y\otimes A(u+v)),}
\]
where the vertical arrows are induced by restriction to $Y$.
By definition, $K_{r-p,0}(X,L;\omega_{X})$ is the cohomology in the middle of the first row.
By Equations (\ref{eqnsectionL}) to (\ref{eqnsectionO}), all three vertical arrows are isomorphisms, thus 
\[
Ker(d_0)\cong Ker(\delta_0)\cong H^0(\wedge^{r-p}M_A\otimes K_Y(u+v)).
\]
and 
the statement follows immediately.
\end{proof}
\begin{corollary}Notation as above, if 
$$h^0(\wedge^{r-p}M_A\otimes K_Y)=h^0(\wedge^{r-p}M_A\otimes K_Y(u+v)),$$ 
or equivalently,
 \begin{eqnarray}\label{eqnRR}
h^0((\wedge^pM_A\otimes A(-u-v))= h^0(\wedge^pM_A\otimes A)-2{r\choose p},
 \end{eqnarray}
 then 
 $$K_{r-p,0}(X,L;\omega_{X})\cong K_{r-p,0}(Y,A;K_{Y}).$$
\end{corollary}
\begin{proof}Immediate. The equivalence of the two assumptions followed from Riemann-Roch and the fact that $\wedge^{r-p}M_A^\lor\cong \wedge^pM_A\otimes A$.
\end{proof}

By degenerating to the pair $(X,L)$, we obtain
\begin{theorem}\label{thminduction}Suppose a general pair $(Y,A)$ in $\cW^r_{g,d}$ satisfies one of the two conditions:
\begin{enumerate}
\item $K_{p,1}(Y,A)=0$;
\item $K_{p-1,2}(Y,A)=0$ and the vector bundle $\wedge^pM_A\otimes A$ satisfies (\ref{eqnRR}) for some $u,v\in Y$.
\end{enumerate}
Then the property ${\bf GV}(p)^r_{g+1,d+1}$ holds.
\end{theorem}

\section{The case $p=1$}
In the case $p=1$, Equation (\ref{eqnRR}) has a very geometric interpretation.

\begin{lemma}\label{lemquadric}For a pair $Y\stackrel{\phi_{|A|}}\hookrightarrow\bP^r$ in $\cW^r_{g,d}$, the vector bundle $M_A\otimes A$ satisfies equation (\ref{eqnRR}) for some $u,v\in Y$ if and only if there exists a quadric hypersurface $Q\subset\bP^r$ containing $Y$ but {\bf not} containing its first secant variety $\Sigma_1(Y)$.   
\end{lemma}
\begin{proof} \begin{enumerate}
\item[($\Longleftarrow$)] The "$\ge$" direction of (\ref{eqnRR}) is automatically satisfied. For the other direction,
consider the following diagram with exact rows %(the surjectivity of $\mu$ is equivalent to $K_{0,2}(Y,A)=0$)
$$\xymatrix{0\ar[r]&H^0(M_{A}\otimes A(-u-v))\ar[r]\ar@{^{(}->}[d]&H^0(A)\otimes H^0(A(-u-v))\ar[r]^-{\mu'}\ar@{^{(}->}[d]&H^0(A^2(-u-v))\ar@{^{(}->}[d]\\
0\ar[r]&H^0(M_{A}\otimes A)\ar[r]&H^0(A)\otimes H^0(A)\ar[r]^-{\mu}&H^0(A^2).}$$

We need to show
$$\dim_{\bC}Ker(\mu')\le\dim_{\bC}Ker(\mu)-2r.$$

Denote $H_{u,v}:=H^0(A)\otimes H^0(A(-u-v))$ and $\overline{H}_{u,v}$ be its image in 
$$\frac{H^0(A)\otimes H^0(A)}{\wedge^2H^0(A)}\cong S^2H^0(A).$$
Note that
\[
\overline{H}_{u,v}\cong \frac{H_{u,v}}{H_{u,v}\cap\wedge^2H^0(A)},
\]
is the space of quadrics which contain the secant line $\overline{uv}$. Thus
\[\dim_{\bC}\overline{H}_{u,v}={r+2\choose2}-3.
\]

We claim that 
\[H_{u,v}\cap\wedge^2H^0(A)=\wedge^2H^0(A(-u-v)).
\]
This is because
\begin{eqnarray}
&&\dim_\bC H_{u,v}\cap\wedge^2H^0(A)\nonumber\\
&=&\dim_\bC H_{u,v}-\dim_\bC \overline{H}_{u,v}\nonumber\\
&=&(r+1)(r-1)-[{r+2\choose2}-3]={r-1\choose2}\nonumber\\
&=&\dim_\bC \wedge^2H^0(A(-u-v)).\nonumber
\end{eqnarray}
The claim is proved.

By hypothesis, $\overline{Ker(\mu)}\nsubseteq \overline{H}_{u,v}$ for some $u,v\in Y$ (since $Q\notin \overline{H}_{u,v}$), then it follows that

$$\dim_{\bC}(\overline{Ker(\mu')})=\dim_{\bC}(\overline{Ker(\mu)\cap H_{u,v}})\le\dim_{\bC}(\overline{Ker(\mu)}\cap\overline{H}_{u,v})\le\dim_{\bC}(\overline{Ker(\mu)})-1=:m-1.$$
Thus
\begin{eqnarray}
\dim_{\bC}(Ker(\mu'))&\le& m-1+\dim_{\bC}(\wedge^2 H^0(A)\cap H_{u,v})\nonumber\\&=&m-1+\dim_{\bC}(\wedge^2H^0(A(-u-v)))\nonumber\\&=&m-1+{r-1\choose2}=m+{r+1\choose2}-2r\nonumber\\&=&\dim_{\bC}(Ker(\mu))-2r.\nonumber
\end{eqnarray}
\item[($\Longrightarrow$)] Reverse the above argument we get the \textquotedblleft only if " part.
\end{enumerate}
\end{proof}
\begin{lemma}\label{lemsecant}Suppose $Y\hookrightarrow\bP^r$ is a nondegenerate curve in $\bP^r$, then there does {\bf not} exist any quadric hypersurface $Q$ containing $\Sigma_1(Y)$. 
\end{lemma}
\begin{proof}Suppose $Y\subset\Sigma_1(Y)\subset Q$ for some quadric $Q$. Fix a point $u\in Y$, since $Q$ contains $\Sigma_1(Y)$, $Q$ must contain the variety $\cJ(u,Y)$ of lines joining $u$ and $Y$. Thus the quadric $Q$ is singular at $u$. (If $Q$ is smooth at $u$, a secant line $\overline{uw}\subset Q$ if and only if $\overline{uw}\subset T_uQ$. Choose $w\in Y\setminus T_uQ$, we have $\overline{uw}\subset \cJ(u,Y)\subset  \Sigma_1(Y)$ but $\overline{uw}\nsubseteq Q$.) Since $u$ is chosen arbitrarily, we conclude that $Q$ is singular along $Y$. This is absurd since the singular locus of a quadric is a linear subspace in $\bP^r$ which can not contain the nondegenerate curve $Y$.
\end{proof}

\begin{proof} of {\bf Theorem \ref{thmmain}}. Follows immediately from Theorem \ref{thminduction} and Lemmas \ref{lemquadric}, \ref{lemsecant}.
\end{proof}

\section{Applications to the Maximal Rank Conjecture}

As an application of Theorem \ref{thmmain}, we obtain a proof of Theorem \ref{thmapp}. 

We say a triple $(g,r,d)$, or equivalently $(g,r,h^1=g-d+r)$ is a base case for the MRC if the Brill-Nother number $\rho:=g-(r+1)h^1=0$. 
\begin{theorem}I\label{thmbase}f the MRC holds for all $\rho=0$ cases, then it holds for arbitrary $\rho\ge0$ case.
\end{theorem} 
\begin{proof}Apply Theorem \ref{thmmain} and induction. Start with any base case, for which we assume property ${\bf GV}(1)^r_{g,d}$ holds. In each step of the induction, $r$ and $h^1$ is fixed and $g$ (equivalently $\rho$ or $d$) goes up by $1$.
\end{proof}

The MRC for the base cases are known to be true when $h^1\le2$.  According to the value of $h^1$, we have the following.
\begin{enumerate}
\item $h^1=0$. We have $g=0$ and $d=r\ge1$, i.e. $(Y,A)=(\bP^1,\cO_{\bP^1}(d))$. The rational normal curves are projectively normal.
\item $h^1=1$. In this case, $g=r+1, d=2r$, i.e $(Y,A)=(Y,K_Y)$. By Nother's theorem, canonical curves are projectively normal ($r\ge2$).
\item $h^1=2$. Then $g=2r+2, d=3r$. Such pairs $(Y,A)$ are projectively normal for $r\ge4$ is the main result of \cite{W1} (The MRC is easy to check when $r=2$ or $3$).
\end{enumerate}

Farkas \cite{F} also proved that ${\bf GV}(1)^{2s}_{s(2s+1),2s(s+1)}$ holds for any $s\ge1$. This covers the base cases when $h^1=s$ and $r=2s$.
\begin{proof} of {\bf Theorem \ref{thmapp}}. Follows immediately from the base cases with $h^1\le2$ and Theorem \ref{thmbase}. \end{proof}

\end{document}